\documentclass[11pt,a4paper]{article}
\usepackage[OT1]{fontenc}
\usepackage[english]{babel}
\usepackage{amsfonts}
\usepackage{amsthm}
\usepackage{color}

\def\aa{{\bf a}}
\def\bb{{\bf b}}

\def\ee{{\bf e}}

\def\TT{\mathbb{T}}
 
\def\d{\displaystyle}
\def\h{ {\cal H} }
\def\a{ {\cal A} }
\def\l{ {\cal L} }

\def\b{ {\cal B} }
\def\u{ {\cal U} }

\def\t{ {\cal T} }
\def\s{ {\cal S} }

\def\p{ {\cal P} }

\def\k{ {\cal K} }

\def\sch{ {\bf  Sd} }
\def\ind{ \hbox{ind} }

\newtheorem{teo}{Theorem}[section]
\newtheorem{prop}[teo]{Proposition}
\newtheorem{lem}[teo]{Lemma}
\newtheorem{coro}[teo]{Corollary}

\theoremstyle{definition}
\newtheorem{rem}[teo]{Remark}
\newtheorem{ejem}[teo]{Example}
\newtheorem{ejems}[teo]{Examples}

\title{Unitary operators with decomposable corners}
\author{Esteban Andruchow}

\begin{document}

\maketitle 

\begin{abstract}
We study pairs $(U,\l_0)$, where $U$ is a unitary operator in $\h$ and $\l_0\subset \h$ is a closed subspace, such that
$$
P_{\l_0}U|_{\l_0}:\l_0\to\l_0
$$
has a singular value decomposition. Abstract characterizations of this condition are given, as well as relations to the geometry of projections and pairs of projections. Several concrete examples are examined. 
\end{abstract}

\bigskip



\section{Introduction}
In this paper we consider pairings $(U,\l_0)$ of a unitary operator $U$ in a Hilbert space $\h$ and a closed subspace $\l_0\subset \h$ such that
$$
P_{\l_0} U|_{\l_0}:\l_0\to\l_0
$$
admits a singular value decomposition  (or shortly, is $S$-decomposable, meaning {\it Schmidt decomposable}). Here $P_{\l_0}$ denotes the orthogonal projection onto $\l_0$. Note  that this condition is equivalent  to say that $P_{\l_0}UP_{\l_0}$ is $S$-decomposable.  A typical case of this situation occurs when $\l_0$ is an invariant subspace for $U$: in this case $P_{\l_0}U|_{\l_0}=U|_{\l_0}$ is an isometry. 

There is a spatial characterization of this condition (see Corollary \ref{equivalentes} below): $P_{\l_0}U|_{\l_0}$  is $S$-decomposable  if and only if there exist bi-orthonormal bases  of $\l_0$ and  $U\l_0$, i.e.,  bases $\{f_n: n\ge 1\}$ of $\l_0$ and $\{g_n: n\ge 1\}$ of  $U\l_0$ such that $\langle f_n, g_m\rangle=0$ if $n\ne m$.

The problem is related to the characterization of pairs of projections $P,Q$ such that $PQ$ is $S$-decomposable, or equivalently, $PQP$ is diagonalizable. Indeed, $P_{\l_0}UP_{\l_0}$ is $S$-decomposable if and only if $P_{\l_0}(UP_{\l_0}U^*)$ is $S$-decomposable.

We shall establish characterizations and abstract results concerning these pairings $(U,\l_0)$:
\begin{itemize}
\item
Relations with the geometry of the Grassmann manifold of $\h$: when  does the exponential map of the manifold $e^{iZ}\l_0$  at a base point $\l_0$ give rise to a  $S$-decomposable operator $P_{\l_0} e^{iZ}|_{\l_0}$ (Section 5). 
\item
Symmetries $U$ (i.e. $U^*=U^{-1}=U$) which have this property with respect to $\l_0$. In particular, symmetries which arise from non-orthogonal projections (Section 6).
\item
The relationship with diagonalizable dilations (Section 7).
\end{itemize}
But also our interest will be in several concrete examples:
\begin{itemize}
\item
Multiplication by continuous unimodular functions in $\h=L^2(\mathbb{T})$ and $\l_0=H^2(\mathbb{T})$. 
\item
$\h=L^2(\mathbb{R})$, $U$ the Fourier-Plancherel transform and $\l_0=L^2(I)$, where $I$ is an interval or the half line.
\item 
$\h=\ell^2(\mathbb{Z})$ and $U=S$ the bilateral shift, $\l_0\subset\ell^2(\mathbb{Z})$ a closed subspace.  
\end{itemize}

The contents of the paper are the th following:

In Section 2 we recall preliminaries and establish basic properties. We denote the fact that $P_{\l_0}U|_{\l_0}$ is $S$-decomposable by writing, equivalently,
$$
U\in\sch_{\l_0} \ \hbox{ or } \ \l_0\in\sch^U,
$$
depending on the standpoint. We also introduce the main examples.

In Section 3 we fix the unitary operator and consider properties of the subspaces $\l\in\sch^U$. For instance, we show that $\sch^U$ is closed for the operation of taking orthogonal supplements, but the orthogonal sum of two subspaces in $\sch^U$ may fail to remain in $\sch^U$.

In Section 4 we give another equivalent condition for $P_{\l_0}U|_{\l_0}$ to be $S$-decomposable in terms of commutators.

In Section 5 we study the relation of this condition with the geometry of the Grassmann manifold of $\h $; specifically,  with the geodesics and exponential map of this manifold.

A non orthogonal projection $Q$, via the polar decomposition
$$
2Q-1=\rho_Q |2Q-1|
$$
gives rise to a symmetry $\rho_Q$ (see \cite{cpr}). We characterize when $P_{R(Q)} \rho_Q|_{R(Q)}$ is $S$-decomposable. This is done in Section 6.

In Section 7 we characterize contractions which are $S$-decomposable, in terms of diagonalization properties of their unitary dilations.

In Section 8 we return to the example $U=M_\varphi$ for $\varphi$ a continuous unimodular function in $\mathbb{T}$, $\h=L^2(\mathbb{T}$ and $\l_0=H^2(\mathbb{T})$. This is an $S$-decomposable pairing: here 
$P_{\l_0}U|_{\l_0}$ is the Toeplitz operator $T_\varphi$, which has a singular value decomposition: it gives rise to a sequence (of singular values) which converges to $1$. We think that this is an interesting fact that needs to be studied. We merely scratch the surface here, examining the case when $\varphi$ is a quotient of finite Blaschke products.

Let us finish this introduction by recalling  the Halmos decomposition of $\h$ relative to a pair of projections / subspaces.
Given projections $P$ and $Q$, put 
$$
\h_{11}=R(P)\cap R(Q) , \ \h_{00}=N(P)\cap N(Q) , \ \h_{10}=R(P)\cap N(Q) , \ \h_{01}=N(P)\cap R(Q) ,
$$
and 
$$
\h'=\{\h_{11}\oplus\h_{00}\oplus\h_{10}\oplus\h_{01}\}^\perp.
$$
The last subspace is usually called the {generic} part of $P$ and $Q$.
Clearly these five subspaces reduce simultaneously $P$ and $Q$. 
For the generic part, in \cite{halmos} Halmos proved that there exists a
unitary isomorphism $\h'\simeq\l\times\l$ such that in this product space,
the reductions $P'$ and $Q'$ of  $P$ and $Q$ to $\l\times \l$ are of the form
$$
P'=\left(\begin{array}{cc} 1 & 0 \\ 0 & 0 \end{array} \right) \ \hbox{ and }  \ \ Q'=\left(\begin{array}{cc} C^2 & CS \\ CS & S^2 \end{array} \right), 
$$
where $C=\cos(X)\ge 0 $ and $S=\sin(X)\ge0$ for $\pi/2\ge X\ge 0$; the three operators have trivial nullspaces, and clearly commute.    
\section{Preliminaries}
Let $\h$ be a Hilbert space, and $\l_0\subset\h$ a closed subspace.
Denote by $P_0=P_{\l_0}$ the orthogonal projection onto $\l_0$.
Let $\u(\h)$ be the unitary group of $\h$.
An operator $T$ acting in a Hilbert space $\h$ is said to be
{\it Schmidt decomposable} in $\h$, if it has a singular value decomposition:
there exist orthonormal system $\{f_n: n\ge 1\}$ and $\{g_n: n\ge 1\}$,
and positive numbers $s_n=s_n(T)$ such that
$$
T=\sum_{n\ge 1} s_n f_n\otimes g_n,
$$
where, as is usual notation, $f\otimes g$ is the rank one operator
$f\otimes g (h)=\langle h, g \rangle f$.
In this note we study the set 
$$
\sch_{\l_0}:=\{U\in\u(\h): P_0 U |_{\l_0} \hbox{ is Schmidt decomposable in } \l_0\}.
$$ 

Note that $U\in\sch_{\l_0}$ if and only if $P_0UP_0$ is Schmidt decomposable.

Let us recall the following facts on Schmidt decomposable products of projections,
taken from \cite{pqvsp-q}.

\begin{prop} {\rm (\cite{pqvsp-q})}\label{pq}
Let $P,Q$ be orthogonal projections.
\begin{enumerate}
\item\label{itempq}
$PQ$ is Schmidt decomposable if and only if there exist orthonormal
bases
$\{\psi_n: n \geq 1\}$ of $R(P)$ and 
$\{\xi_n : n \geq 1\}$ of $R(Q)$ such that 
$\langle \xi_n,\psi_k\rangle=0$ if $n\ne k$.
In that case, $PQ = \d\sum_{n \geq1} s_n \psi_n \otimes \xi_n$
where $s_n = \langle \xi_n, \psi_n \rangle$ are the singular values of $PQ$.
Moreover, $s_n \leq 1$ and for all $n$ such that $s_n = 1$ the
associated vectors $\xi_n$ and $\psi_n$ verify that $\xi_n = \psi_n$
and generate $R(P) \cap R(Q)$.
\item
$PQ$ is Schmidt decomposable if and only if $P-Q$ is diagonalizable.
In that case, if $s_n$ are the singular values of $PQ$, then
the eigenvalues of $P-Q$ are $\pm (1-s_{n}^{2})^{1/2}$, $n\geq 1$,
and eventually, $0, -1$ and $1$.
\item
$PQ$ is Schmidt decomposable if and only if $QP$, or $P^\perp Q$,
or $PQ^\perp$ or $P^\perp Q^\perp$ are Schmidt decomposable.
Moreover, the singular values $s_n$ of $PQ$ and $t_n$ of $PQ^\perp$,
such that $s_n,t_n<1$,  are related by
$$
t_n=\sqrt{1-s_n^2},
$$
with the same multiplicity. 
In particular, $PQ$ and $P^{\perp} Q^{\perp}$ have the same singular values
(which are strictly less than one), with the same multiplicity.

\end{enumerate}
\end{prop}

These facts have the following immediate consequences:

\begin{coro}\label{equivalentes}
Let $U\in\u(\h)$. The following are equivalent:
\begin{enumerate}
\item
$U\in \sch_{\l_0}$.
\item
$WP_0 U P_0V^*$ is Schmidt decomposable, for all
$V,W\in\u(\h)$.

In particular, $U\in\sch_{\l_0}$ if and only if $P_0UP_0U^*$ is Schmidt decomposable,
and this latter operator is a product of projections.
\item
The commutator $[P_0,U]=P_0U-UP_0$ is Schmidt decomposable.
\item
$U\in\sch_{\l_0^\perp}$.
\item
There exist orthonormal basis $\{f_n\}$ and $\{f'_n\}$  of $\l_0$ such that
$$
\langle f_n,Uf'_m\rangle=0 \ \ \hbox{ if } n\ne m.
$$ 
\end{enumerate}
In this case, the singular values of $P_0U|_{\l_0}$ are $s_n=|\langle f_n,Uf'_n\rangle|$,
and the singular values of $[P_0,U]$ are the absolute values of the
eigenvalues of $P_0 - U P_0 U^{\ast}$.
\end{coro} 

\begin{rem}
Note that $U\in\sch_{\l_0}$ implies that also
$U^*\in\sch_{\l_0}$: if $P_0UP_0=\sum_{n\ge 1} s_n\xi_n\otimes \eta_n$, then $P_0U^*P_0=(P_0UP_0)^*=\sum_{n\ge 1} s_n\eta_n\otimes \xi_n$. Also it is apparent that $I\in\sch_{\l_0}$. But $\sch_{\l_0}$ is not a group, as the following example shows. Let $P,Q$ be projections in a Hilbert space $\l$ such that $PQ$ is not Schmidt decomposable, (see below for an for explicit example). Let $\h=\l\times\l$ and $\l_0=\l\times\{0\}$. Consider in $\h$ the unitary operators
$$
U_P=\left( \begin{array}{cc} P & 1-P \\ 1-P & P \end{array} \right) \ \hbox{ and } \ U_Q=\left( \begin{array}{cc} Q & 1-Q \\ 1-Q & Q \end{array} \right).
$$
Clearly $P_0U_PP_0=\left( \begin{array}{cc} P & 0 \\ 0 & 0 \end{array} \right)$ is Schmidt decomposable, and the same for $U_Q$. But
$$
P_0U_PU_QP_0=\left( \begin{array}{cc} PQ+(1-P)(1-Q) & 0 \\ 0 & 0 \end{array} \right),
$$
which we claim is non decomposable. Indeed, if $T=PQ+(1-P)(1-Q)$ were decomposable, $TT^*=PQP+(1-P)(1-Q)(1-P)$ would be diagonalizable, which would imply in particular that $PQP$ is diagonalizable, and thus $PQ$ would be decomposable.
\end{rem}
\begin{ejems}\label{ejemplos}

\begin{enumerate}
\noindent
\item
Let $\h=L^2(\mathbb{T})$ and $\l_0=H^2(\mathbb{T})$ the Hardy space. Let $\varphi:\mathbb{T}\to\mathbb{T}$ be continuous. Then the multiplication (unitary) operator $M_\varphi\in\sch_{\l_0}$. Indeed, $P_{\l_0}M_\varphi P_{\l_0^\perp}$ is a Hankel operator with continuous symbol, thus by Hartman's theorem \cite{hartman} it is compact, and thus Schmidt decomposable. Then $P_{\l_0}M_\varphi P_{\l_0^\perp} M_{\bar{\varphi}}$ and also 
$$
P_{\l_0}(1-M_\varphi P_{\l_0^\perp} M_{\bar{\varphi}})=P_{\l_0}M_\varphi P_{\l_0} M_{\bar{\varphi}}
$$
are Schmidt decomposable, as well as $P_{\l_0}M_\varphi P_{\l_0}$. Note that the same argument holds for $\varphi$ a unimodular function in $C(\mathbb{T})+H^\infty(\mathbb{T})$.
\item
The previous example can be generalized to an abstract setting. Let $\l_0\subset\h$ of infinite dimension and co-dimension. The restricted unitary group (relative to the decomposition $\h=\l_0\oplus\l_0$) is defined as
$$
\u_{res}(\l_0)=\{U\in\u(\h): [U,P_0] \hbox{ is compact}\}.
$$
Note that it is the unitary group of the C$^*$-algebra $\a_{\l_0}=\{A\in\b(\h): [A,P_0] \hbox{ is compact}\}$.
Also, if  the matrix of $U$ in terms of this decomposition is $U=\left(\begin{array}{cc} U_{11} & U_{12} \\ U_{21} & U_{22}\end{array}\right)$, then $[U,P_0]$ compact means that $U_{12}$ and $U_{21}$ are compact. Then (using that $U$ is unitary), $U_{11}U_{11}^*+U_{12}U_{12}^*=1$, i.e. $U_{11}U_{11}^*=1+K$ with $K$ compact. This implies that $U_{11}$ is Schmidt decomposable, that is 
$$
\u_{res}(\l_0)\subset \sch_{\l_0}.
$$
Also it is clear that $U_{11}$ (as well as $U_{22}$) is a Fredholm operator. The connected components of $\u_{res}(\l_0)$ are parametrized by the Fredholm index of the $1,1$ entry.

Clearly, $M_\varphi$ of example 1 belongs to $\u_{res}(H^2(\TT))$. The index (of the $1,1$ entry coincides with minus  the winding number of $\varphi$). See for instance \cite{segalwilson}.
The connected component of the identity contains the often called Fredholm unitary group $\u_\infty(\h)=\{U\in\u(\h): U-1 \hbox{ is compact}\}$. 

\item
Let $\h=L^2(\mathbb{R}^n)$, $\Omega\subset\mathbb{R}^n$ be a measurable set with $|\Omega|<\infty$, and put $\l_0=L^2(\Omega)$ (considered as a closed subspace of $\h$). Let $U$ be the Fourier-Plancherel transform. Then $U\in\sch_{\l_0}$. Indeed, $P_{\l_0}U^*P_{\l_0}U$ is the composition of the projections onto (respectively) the Lebesgue space $L^2(\Omega)$ and the  Wiener space $W(\Omega)$ of  $\Omega$ (i.e.,  $W(\Omega)=\{f\in L^2(\mathbb{R}): \hat{f}|_{\Omega^c}=0 \hbox{ a.e.}\}$). It is known (see for instance \cite{landau}, or the survey article \cite{folland}) that this composition is of trace class, thus decomposable. Moreover, if $\theta, \omega$ are measurable  unimodular functions in $\mathbb{R}^n$, then
$M_\theta U M_\omega\in\sch_{\l_0}$: $M_\theta, M_\omega$ commute with $P_{\l_0}$, and thus
$$
P_{\l_0}M_\theta U M_\omega P_{\l_0}=M_\theta P_{\l_0} U  P_{\l_0}M_\omega
$$
is decomposable.
Also note that $U\in\sch_{W(\Omega)}$:
$$
P_{W(\Omega)}UP_{W(\Omega)}=U^{-1}P_{L^2(\Omega)}UUU^{-1}P_{L^2(\Omega)}U=U^{-1}P_{L^2(\Omega)}UP_{L^2(\Omega)}U
$$
which is Schmidt decomposable.

\item
One can characterize the symmetries (i.e., selfadjoint unitaries) which belong to $\sch_{\l_0}$. Let $\epsilon_\s$ be the symmetry which is equal to $1$ in $\s$ and $-1$ in $\s^\perp$, i.e.,  $\epsilon_\s=2 P_\s-1$. Then $\epsilon_\s\in\sch_{\l_0}$ if and only if $2P_0P_\s P_0-P_0$ is decomposable, and since it is selfadjoint, diagonalizable. This is clearly equivalent to $P_0P_\s P_0$ being diagonalizable, or $P_0P_\s$ being Schmidt decomposable. 

Consider, for instance $\h=L^2(-1,1)$. Let $U$ be the symmetry $Uf(t)=f(-t)$, and $A$ the selfadjoint (non diagonalizable) contraction $Af(t)=tf(t)$. Note that $UAU=-A$. Chandler Davis \cite{davis} gives in this case formulas for pairs of projections  $P_U$, $Q_U$ satisfying $UP_UU=Q_U$ and $P_U-Q_U=A$. Namely
$$
P_U=\frac12\{1+A+U(1-A^2)^{1/2}\} \ \hbox{ and } \  Q_U=\frac12\{1-A+U(1-A^2)^{1/2}\}.
$$
Then, by Proposition \ref{pq}, since $A=P_U-Q_U$ is not diagonalizable, $P_UQ_U=P_UUP_UU$ is not Schmidt decomposable, i.e. $U\notin\sch_{R(P_U)}$. Note that 
$$
R(P_U)=\{f\in L^2(-1,1): f(-t)=f(t)((1-t^2)^{1/2}-t)\  \hbox{a.e.}\}.
$$
\end{enumerate}
\end{ejems}

\section{Fixing the unitary operator}

We  fix a unitary operator $U$ in $\h$, and consider the set of closed subspaces $\l\subset\h$ such that $U\in\sch_\l$:
\begin{equation}\label{schU}
\sch^U:=\{\l\subset\h: \l \hbox{ is closed and } U\in\sch_\l\}.
\end{equation}

Let us state the following elementary properties of $\sch^U$.
\begin{prop}\label{31}
Let $U$ be a unitary operator in $\h$. 
\begin{enumerate}
\item
If $\l\in\sch^U$, then $\l^\perp\in\sch^U$.
\item
If $\l$ is an invariant subspace for $U$, then $\l\in\sch^U$.
\item
As a consequence of 1) and 2), if $\l$ is invariant for $U^*$, then $\l\in\sch^U$.
\end{enumerate}
\end{prop} 
\begin{proof}
If $\l\in\sch^U$, then $P_\l U P_\l U^*$ is decomposable. Then 
$$
P_\l^\perp (UP_\l U^*)^\perp=P_{\l^\perp} U P_{\l^\perp}U^*
$$
is decomposable, i.e. $\l^\perp \in\sch^U$.

If $\l$ is an invariant subspace for  $U$, then $U|_\l:\l\to\l$ is an isometry, and thus has a singular value decomposition.

Finally, if $\l$ is invariant for $U^*$, then $\l^\perp$ is invariant for $U$. Thus, $\l^\perp\in\sch^U$, and therefore $\l\in\sch^U$.
\end{proof}
Thus, $\sch^U$ contains the lattice of invariant subspaces of $U$. It is not, however, itself a lattice, as the following remark shows.

\begin{rem}
If $\l_1,\l_2\in\sch^U$, and $\l_1\perp\l_2$, then $\l_1\oplus\l_2\in\sch^U$
may not lie in $\sch^U$.
Indeed, let $P_1,P_2$ be the orthogonal projections onto $\l_1$ and $\l_2$.
Denote by $Q_1=UP_1U^*$,  We want to study if $(P_1+P_2)U(P_1+P_2)$,
or equivalently if $(P_1+P_2)(Q_1+Q_2)$, is decomposable.
As we shall see below, we only need to examine the generic part
of the pair $P_1+P_2, Q_1+Q_2$.
Thus we can suppose $\h=\l\times\l$, and
$$P_1+P_2=
\left(\begin{array}{cc} 1 & 0 \\ 0 & 0 \end{array}\right)\!\!,
\ \  
Q_1+Q_2=
\left(\begin{array}{cc} C^2 & CS \\ CS & S^2 \end{array}\right)=
\left(\begin{array}{cc} C & -S \\ S & C \end{array}\right)\left(\begin{array}{cc} 1 & 0 \\ 0 & 0 \end{array}\right)\left(\begin{array}{cc} C & S \\ -S & C \end{array}\right)\!\!,
$$
where $\left(\begin{array}{cc} C & -S \\ S & C \end{array}\right)$ is a unitary operator. Then
$$ P_1=
\left(\begin{array}{cc} E & 0 \\ 0 & 0 \end{array}\right)\!\!,
\\\
 P_2=\left(\begin{array}{cc} 1-E & 0 \\ 0 & 0 \end{array}\right)\!\!,
\\\
Q_1=\left(\begin{array}{cc} CEC & CES \\ SEC & SES \end{array}\right)\!\!,
\\\
Q_2=\left(\begin{array}{cc} CEC & CES \\ SEC & SES \end{array}\right)\!\!.$$
The assumption that $\l_1,\l_2\in\sch^U$ means that
$P_1Q_1P_1$ and $P_2Q_2P_2$ are diagonalizable in $\l\times \l$, i.e.
$(ECE)^2$ and $((1-E)C(1-E))^2$ are diagonalizable in $\l$, and since $C\ge 0$, $ECE$ and $(1-E)C(1-E)$ are diagonalizable. On the other hand, we have to examine weather these
 assumptions imply that $(P_1+P_2)(Q_1+Q_2)(P_1+P_2)$ is diagonalizable in $\l\times\l$, which is clearly equivalent to $C^2$, or $C$, being diagonalizable. Therefore it suffices to exhibit an example of a positive injective contraction $C$ and a projection $E$, such that $ECE$ and $(1-E)C(1-E)$ are diagonalizable, but $C$ is not.

Consider for instance
$C=\frac12\left(\begin{array}{cc} 1 & A \\ A & 1 \end{array}\right)$
acting in $L^2(0,1)\times L^2(0,1)$, and $A=M_t$
(multiplication by the variable $t$) in $L^2(0,1)$.
Then it is easy to see that $C$ is a positive contraction with trivial nullspace.
Also, if  $E=\left(\begin{array}{cc} 1 & 0 \\ 0 & 0 \end{array}\right)$,
it is clear that $ECE=(1-E)C(1-E)=1$ in $L^2(0,1)$.
But $C$ is not diagonalizable.
If it were,
$C - \textstyle\frac{1}{2} 1 = 
\left(\begin{array}{cc} 0 & A \\ A & 0 \end{array}\right)$
would be diagonalizable, and thus
$$\left(\begin{array}{cc} 0 & A \\ A & 0 \end{array}\right)^2=
\left(\begin{array}{cc} A^2 & 0 \\ 0 & A^2 \end{array}\right)$$
and therefore $A^2$, and $A$ would be diagonalizable.
\end{rem}

Let us exhibit an example of a closed subspace which does not belong to $\sch^S$, where $S$ is the bilateral shift operator acting in $\ell^2(\mathbb{Z})$.
\begin{ejem}\label{ejem 32}
Let $S$ be the bilateral shift operator in $\ell^2(\mathbb{Z})$. Consider the closed subspace
$$
\l_0=\{(a_k)\in\ell^2(\mathbb{Z}): a_k=a_{-k} \hbox{ for all } k\in\mathbb{Z}\}.
$$
Denote by $\Pi\in\b(\ell^2(\mathbb{Z}))$ the symmetry $\Pi(a_k)_m=a_{-m}$. Then it is elementary that
$P_{\l_0}=\frac12(1+\Pi)$. Thus $P_{\l_0} S P_{\l_0}=\frac14 (1+\Pi)S(1+\Pi)$. Denote by $\ee_n\in \ell^2(\mathbb{Z})$ the elements of the canonical basis of $\ell^2(\mathbb{Z})$. Note that
$$
P_{\l_0}SP_{\l_0}\ee_0=\frac12(\ee_1+\ee_{-1}) \ , \ \  P_{\l_0}SP_{\l_0}\ee_{\pm 1}=\frac14(\ee_2+\ee_{-2}+2\ee_0)
$$ 
and 
$$
P_{\l_0}SP_{\l_0}\ee_n=\frac14(\ee_{n+1}+\ee_{n-1}+\ee_{-n-1}+\ee_{-n+1})
$$
for $n\ne 0, \pm 1$.
Then, after another elementary computation,
$$
\langle P_{\l_0}S P_{\l_0}\ee_n , \ee_m\rangle =\langle \ee_n , P_{\l_0}S P_{\l_0}\ee_m \rangle,
$$
for all $n.m\in\mathbb{Z}$. It follows that $P_{\l_0}S P_{\l_0}$, regarded as an operator in $\ell^2(\mathbb{Z})$, is selfadjoint. Thus, if it where decomposable, it would be diagonalizable. Let us show that it has no eigenvectors. We  identify $\ell^2(\mathbb{Z})$ with $L^2(\TT,\frac{dz}{2\pi})$ with the usual isomorphism, which carries $\ee_n$ to $z^n$ ($n\in\mathbb{Z}$). Then the subspace $\l_0$ is given by
$$
\l_0=\{f\in L^2(\TT): f(z)=f(\bar{z}) \hbox{ a.e.}\},
$$
the symmetry $\Pi$ is $\Pi f(z)=f(\bar{z})$, and   $P_{\l_0}SP_{\l_0} f(z)=\frac14 (z+\bar{z})(f(z)+f(\bar{z}))$. Then
$$
P_{\l_0}S |_{\l_0} f(z)=\frac12 (z+\bar{z})f(z) \ , \ \ \hbox{ for all } f\in\l_0.
$$
If $g\in\l_0$ were an eigenvector for this operator, then $\lambda g(z)=\frac{1}{2}(z+\bar{z})g(z)$ a.e., and thus $g=0$. 
\end{ejem}

 On the other hand, model spaces do belong to $\sch^S$:
\begin{ejem}
Let $\theta$ be an inner function in $\TT$. Consider the model space $\l_0=\k_\theta=H^2(\TT)\ominus \theta H^2(\TT)$,  here regarded as a subspace of $\h=L^2(\TT)$, and let $S\in\b(L^2(\TT))$ be again the bilateral shift operator, $Sf(z)=zf(z)$. Let us show that 
$$
\k_\theta\in\sch^S.
$$
We shall work with $\k_\theta^\perp=H_2(\TT)^\perp\oplus \theta H_2(\TT)$. Let us denote $\h^-=H_2(\TT)^\perp$ and $\h_\theta=\theta H_2(\TT)$. Note that $P_\theta:=P_{\h_\theta}=M_\theta P_+ M_{\bar{\theta}}$, where $P_+=P_{H_2(\TT)}$ (accordingly $P_-=1-P_+$). Then $\k_\theta^\perp\in\sch^S$ if and only if the operator 
$$
(P_- +P_\theta)S(P_- +P_\theta)=P_-SP_-+P_\theta SP_\theta +P_-SP_\theta+ P_\theta SP_-
$$
is Schmidt decomposable. First note that $P_-SP_\theta=0$: $S(R(P_\theta))=S(\theta \h^+)\subset \h^+$. The other off-diagonal operator, $P_\theta S P_-$ has rank one.
Indeed, if $f=\sum_{n\in\mathbb{Z}}z^n\in L^2(\TT)$, let $f_-=P_-f$. Then
$$
P_\theta SP_-f=P_\theta zf_-(z)=P_\theta (a_{-1}1)=a_{-1}P_\theta(1)=a_{-1}M_\theta P_+\bar{\theta}=a_{-1}\bar{\theta}(0) \theta=\theta(0) \langle f, z^{-1}\rangle \theta,
$$
i.e. $P_\theta SP_-=\overline{\theta(0)} \theta\otimes z^{-1}$. The other (diagonal) entries are
$$P_-SP_-,
$$
which is a co-isometry in $\h_-$, whose adjoint $P_-S^*P_-$ is an isometry with range $S^*(\h_-)=z^{-1}\h_-$; and
$$
P_\theta SP_\theta,
$$
which is an isometry in $\h_\theta$ with range $z \h_\theta$. We can write almost explicitly a Schmidt decomposition for $(P_- +P_\theta)S(P_- +P_\theta)$: 
$$
P_-SP_-=\sum_{m<0} z^m\otimes z^{m-1}.
$$
If we consider the Lebesgue measure normalized in $\TT$, then $\theta$ is a unit vector in  $\h_\theta$. Let $\{f_n\}_{n\ge 1}$ be an orthonormal basis for $\h_\theta$ with $f_1=\theta$. Then $P_\theta S P_\theta=\sum_{n\ge 1} zf_n\otimes f_n$.
Then
\begin{equation}\label{desc S}
(P_- +P_\theta)S(P_- +P_\theta)=\sum_{m<0} z^m\otimes z^{m-1}+\sum_{n\ge 1} zf_n\otimes f_n+\overline{\theta(0)}\ \theta\otimes z^{-1}.
\end{equation}
Note that in this expression, $z^{-1}$ is orthogonal to $z^{m-1}$ ($m<0$) and to $f_n$ ($\in\h_\theta$). Also, $\theta$ is orthogonal to $z^m$ ($m<0$) and to $zf_n$. Indeed, $\theta\perp z\theta \h^+$: if $h\in\h^+$
$$
\langle z\theta h,\theta\rangle= \frac{1}{2\pi}\int_\TT z|\theta(z)|^2 h(z)\ dz=\frac{1}{2\pi}\int_\TT z h(z)\  dz=0,
$$
by Cauchy's Theorem. Thus, the above expression (\ref{desc S}) is essentially a singular value decomposition for $(P_- +P_\theta)S(P_- +P_\theta)$. It only remains to normalize the term $\overline{\theta(0)}\ \theta\otimes z^{-1}$: let $\theta(0)=|\theta(0)|e^{i\alpha}$. Then $\overline{\theta(0)}\ \theta\otimes z^{-1}=|\theta(0)| \theta_0\otimes z^{-1}$, where $\theta_0=e^{-i\alpha}\theta$.

Note that the singular values of $(P_- +P_\theta)S(P_- +P_\theta)$ are an infinite list of $1$'s, and the number $|\theta(0)|$. Thus, using Remark \ref{pq}, the singular values of 
$$
P_{\k_\theta}SP_{\k_\theta}=(P_- +P_\theta)^\perp S(P_- +P_\theta)^\perp
$$
are also a list of (infinite) $1$'s , and $|\theta(0)|$.
\end{ejem}

\nocite{jot}
\nocite{cpr}
\nocite{davis}
\nocite{pr}
\nocite{schmidt}
\begin{ejem}
In the setting of example \ref{ejemplos}.3 ($\h=L^2(\mathbb{R})$, $U$ the Fourier-Plancherel transform), put $\l_0=L^2(0,+\infty)\subset  \h$. In this case, is $P_0UP_0$ Schmidt-decomposable? 

Denote by $\psi_n(x)=\frac{2^{1/4}}{\sqrt{n\!}}e^{-\pi x^2} {\bf H}_n(x)$ the eigenfunctions of $U$  (where ${\bf H}_n$ is $n$th Hermite polynomial): $U\psi_n=(-i)^n\psi_n$. Since for $n=2k$ even, $\psi_n$ is an even function, it follows that 
$$
(P_0UP_0+P_0U^*P_0) \psi_{2k}=(-1)^kP_0\psi_{2k}.
$$
Indeed, if $x\ge 0$, 
$$
(-1)^k\psi_{2k}(x)=\frac{1}{\sqrt{2\pi}}\int_{-\infty}^\infty \psi_{2k}(t)e^{-ixt}dt=
\frac{1}{\sqrt{2\pi}}\{\int_{-\infty}^0 \psi_{2k}(t)e^{-ixt}dt+\int_{0}^\infty \psi_{2k}(t)e^{-ixt}dt\}
$$
Changing $s=-t$,  the left hand integral becomes $\int_{0}^\infty \psi_{2k}(s)e^{ixs}dt$. 
Note that 
$$
\frac{1}{\sqrt{2\pi}}\int_{0}^\infty \psi_{2k}(t)e^{-ixt}dt=\frac{1}{\sqrt{2\pi}}\int_{-\infty}^\infty \chi_{(0,+\infty)}(t)\psi_{2k}(t)e^{-ixt}dt=\frac{1}{\sqrt{2\pi}}\int_{-\infty}^\infty P_0(\psi_{2k})(t)e^{-ixt}dt
$$
$$
=UP_0\psi_{2k}(x),
$$
and similarly 
$$
\frac{1}{\sqrt{2\pi}}\int_{0}^\infty \psi_{2k}(s)e^{ixs}dt=U^*P_0\psi_{2k}(x).
$$
Clearly $\{P_0\psi_{2k}: k\in\mathbb{N}\}$ is an orthonormal basis for $L^2(0,+\infty)$: if $k\ne k'$,
$$
\langle P_0\psi_{2k},P_0\psi_{2k'}\rangle=\int_0^\infty\psi_{2k}(t)\overline{\psi_{2k'}}(t)dt=\frac12\int_{-\infty}^\infty \psi_{2k}(t)\overline{\psi_{2k'}}(t)dt=0.
$$
A similar argument shows that they span a dense subspace of $L^2(0,\infty)$. Therefore $Re\  P_0UP_0$ is diagonalizable. More specifically $Re\ P_0UP_0=\frac12 P_0(U+U^*)P_0$ is a $\frac12$-times a symmetry, which is the $\frac12$  identity on the subspace spanned by $\{\psi_{4k}: k\ge 1\}$ and  $-\frac12$ the identity in the subspace spanned by $\{\psi_{4k+2}: k\ge 1\}$. 

Similarly, $P_0UP_0-P_0U^*P_0$ is diagonalized, by means of the eigenfunctions $P_0\psi_{2k+1}$ (with eigenvalues $i(-1)^k$), which also form an orthonormal basis of $L^2(0,\infty)$, and thus $Im\ P_0UP_0$ is diagonalizable, and $\frac{i}{2}$ times a symmetry, with a similar description as the real part. 

Then  $P_0 UP_0$ has real and imaginary parts which are diagonalizable. Note that $P_0UP_0$ is not normal, in which case it would be diagonalizable.  If it were normal, then $Re\ P_0UP_0$ and $Im\ P_0UP_0$ would commute, and then
$$
\sigma(P_0UP_0)\subset \sigma(Re \ P_0UP_0) + \sigma(i \ Im \ P_0UP_0)=\{\pm \frac12 \pm \frac{i}{2}\},
$$
and thus $\|P_0UP\|=\frac{\sqrt{2}}{2}$.
On the other hand, let $\chi=\chi_{(0,1)}$
be the characteristic function on the unit interval $(0,1)$.
Then $\|\chi\|_2=1$ and
$$P_0UP_0\chi(x)=\frac{1}{\sqrt{2\pi}}\int_0^1e^{-itx} dt=
\frac{1}{\sqrt{2\pi}}\displaystyle{\frac{1-e^{-ix}}{i \ x}}.$$
Thus, 
$$\|P_0UP_0\chi\|_2^2=
\frac{1}{2\pi}\int_0^\infty 2 \displaystyle{\frac{(1-\cos(x))}{x^2}} dx=1,$$
i.e., since $P_0UP_0$ is clearly a contraction,  $\|P_0UP_0\|=1$.

The question remains, which we consider interesting in its own right, of weather $P_0UP_0$,  the compression of the Fourier transform to the positive half-line, has a singular value decomposition.
\end{ejem}
Recall  Example \ref{ejemplos}.2, were we saw that $\u_{res}(\l_0)\subset  \sch_{\l_0}$. Fix $\l_0$ and $P_0=P_{\l_0}$.  Next we give a sufficient condition for a closed subspace $\l\subset \h$, in order that $\l\in\sch^U$, for all $U\in\u_{res}(\l_0)$.

Recall \cite{ass} that a pair of orthogonal  projections $(P,Q)$ has {\it finite index} if the operator
$$
QP|_{R(P)}:R(P)\to R(Q)
$$
has finite Fredholm index. The index of this operator is called the index $\ind(P,Q)$ of the pair $(P,Q)$. Note that
$$
\ind(P,Q)=\dim(R(P)\cap N(Q))-\dim(N(P)\cap R(Q)).
$$
\begin{prop}
If $\ind(P_\l,P_0)<\infty$, then $\l\in\sch^U$, for all $U\in\u_{res}(\l_0)$.
\end{prop}
\begin{proof}
If $\ind(P_\l,P_0)<\infty$, then there exists $V\in\u_{res}(\l_0)$ such that $P_\l=VP_0V^*$ (see for instance \cite{segalwilson}). Then, if $U\in\u_{res}(\l_0)$,
$$
P_\l UP_\l=VP_0V^*UVP_0V^*.
$$
Note that $V^*UV\in\u_{res}(\l_0)$ (which is a group), and thus (see Example \ref{ejemplos}.2) $P_0V^*UVP_0$ is Schmidt decomposable. 
\end{proof}

\section{Commutators}
Recall from the introduction the Halmos decomposition of $\h$ relative to a pair of projections  $P$ and $Q$:
$$
\h_{11}=R(P)\cap R(Q) , \ \h_{00}=N(P)\cap N(Q) , \ \h_{10}=R(P)\cap N(Q) , \ \h_{01}=N(P)\cap R(Q) ,
$$
and the generic part $\h'\simeq\l\times\l$
$$
\h'=\{\h_{11}\oplus\h_{00}\oplus\h_{10}\oplus\h_{01}\}^\perp.
$$
It is  easy to see that the nullspace of $[P,Q]$ is
\begin{equation}\label{ker}
N([P,Q]) = \h_{11}\oplus\h_{11}\oplus\h_{10}\oplus\h_{01}.
\end{equation}

Our main result in this section states that
$PQ$ is Schmidt decomposable if only if the commutator$[P,Q]$ is diagonalizable.

\begin{teo}\label{conmpq}
Let $P,Q$ be orthogonal projections,
then the following are equivalent:
\begin{enumerate}
\item $PQ$ is Schmidt decomposable,
\item $A= [P,Q] = PQ - QP$ is diagonalizable,
\item $X$ is diagonalizable.
\end{enumerate}
Moreover if
$PQ$ has singular values $s_n$ 
then $[P,Q]$ has eigenvalues $\pm i s_n \sqrt{1-s_{n}^{2}}, n \geq 1$,
and, eventually, $0$. 
\end{teo}

\begin{proof}
If we put
$T= PQ = \d\sum_{n \geq1} s_n \psi_n \otimes \xi_n$,
where $s_n = \langle \xi_n, \psi_n \rangle$ are the singular values of $PQ$
and follow the ideas from \cite[Theorem 2.2]{pqvsp-q},
we get that for all $k$ such that $s_k < 1$,
$$A \xi_k = s_k \psi_k - s_{k}^{2} \xi_k
\hspace{0.7cm}
\textnormal{and}
\hspace{0.7cm}
A \psi_k= s_{k}^{2} \psi_k - s_{k} \xi_k.$$
Then
$$A^{2} \xi_k = (s_{k}^{4}-s_{k}^{2}) \xi_k
\hspace{0.7cm}
\textnormal{and}
\hspace{0.7cm}
A^{2} \psi_k= (s_{k}^{4}-s_{k}^{2}) \psi_k .$$

So
$$v_k = \left( s_{k}^{2}- i s_k \sqrt{1-s_{k}^{2}} \right)  \xi_k  - s_k \psi_k
\hspace{0.7cm}
\textnormal{and}
\hspace{0.7cm}
w_k = \left(s_{k}^{2} + i s_k \sqrt{1-s_{k}^{2}}\right) \xi_k  - s_k \psi_k$$
are orthogonal eigenvectors for $A$, with eigenvalues
$i s_k \sqrt{1-s_{k}^{2}}$ and $-i s_k \sqrt{1-s_{k}^{2}}$,
respectively.

Note that on the extension of the system $\xi_k$, $R(P) \ominus R(T)$, 
and
on the extension of the system $\psi_k$, $R(Q) \ominus N(T)^{\perp}$,
$A$ equals $0$.
On $R(P) + R(Q)$, $A$ is diagonalizable. On the orthogonal complement
of this subspace, namely $N(P)^{\perp} \cap N(Q)^{\perp}$, $A$ is trivial.
%
%
%

\medskip

To prove the converse, we use Halmos decomposition. 
After elementary computations, one sees that
on $\h'=\l\times \l$, $[P,Q]$ is given by
$$
[P',Q']=\left(\begin{array}{cc} 0 & CS \\ -CS & 0 \end{array} \right).
$$
If $[P,Q]$ is diagonalizable, then so are $[P',Q']$ and
$[P',Q']^2=\left(\begin{array}{cc} -C^2S^2 & 0 \\ 0 & -C^2S^2 \end{array} \right)$. Clearly this implies that $C^2S^2$ and its square root $CS$ are diagonalizable. We claim that  $X$ is diagonalizable. Indeed, note that 
$CS=\frac12 \sin(2X)$. The spectrum of $2X$ is contained in $[0,\pi]$.
Let $E=\chi_{[0,\pi/2]}(2X)$, be the spectral projection of
$2X$ corresponding to the interval $[0,\pi/2]$. Then the selfadjoint operator $2XE$, acting in $R(E)$ has spectrum contained in $[0,\pi/2]$, and $2XE^\perp$ acting in $R(E)^\perp$ has spectrum contained in $[\pi/2,\pi]$. Since $E$ and $\sin(2X)$ commute, and both selfadjoint operators are diagonalizable, they can be simultaneously diagonalized: there exist orthonormal vectors $\varphi_n$ such that 
$$
CS=\frac12\sin(2X)=\frac12\sum_{n\ge 1} s_n \varphi_n\otimes \varphi_n,
$$
with $0<s_n<1$, and either $\varphi_n\in R(E)$ or $\varphi_n\in R(E)^\perp$. Then 
$$
CSE=\frac12\sum_{\varphi_j\in R(E)} s_j\varphi_j\otimes\varphi_j \ \hbox{ and } \ CSE^\perp=\frac12\sum_{\varphi_k\in R(E)^\perp} s_k\varphi_k\otimes\varphi_k.
$$
The function  $\arcsin(t)$ is continuous in $\sigma(CSE)\subset  [0,1]$, and one has
$$
2XE=\arcsin(2CSE)=\sum_{\varphi_j\in R(E)} \arcsin(s_j)\varphi_j\otimes\varphi_j.
$$
Then, since $\cos(2XE^\perp-\frac{\pi}{2}E^\perp)=\sin(2XE^\perp)$
and the spectrum  of $2XE^\perp-\pi/2 E^\perp$ is contained in $[0,\pi/2]$ (where $\cos$ has continuous inverse) one has
$$
2XE^\perp-\pi/2E^\perp=\arccos(\sin(2XE^\perp))=\sum_{\varphi_k\in R(E)^\perp} \arccos(s_k)\varphi_k\otimes\varphi_k.
$$
Therefore $X=XE+XE^\perp$ is diagonalizable.
\end{proof}


Thus we may complete Corollary \ref{equivalentes} with the following equivalent conditions:

\begin{coro}
With the current notations, the following are equivalent:
\begin{enumerate}
\item
$U\in\sch_{\l_0}$.
\item
$[P_0,UP_0U^*]$ is diagonalizable.
\item
$X$ is diagonalizable.
\end{enumerate}
\end{coro}
We point out that in figuring out if $U\in\sch_{\l_0}$, only the generic part between $\l_0$ and $U\l_0$ is relevant.
\section{Geodesics of the Grassmann manifold}

Example \ref{ejemplos}.3  is related to the following. Consider the differential geometry of  the Grassmann manifold $\p(\h)$ of $\h$ (see \cite{pr}, \cite{cpr}, or the survey article \cite{survey}). It is known that two   projections/closed subspaces in generic position  can  be joined by a unique minimal geodesic. In other words, for arbitrary $P_{\l_0}, P_\l$, the reductions of these projections to their common generic part can be joined by a unique projections.  On the non generic summands of $\h$, the obstruction for the existence of a geodesic joining two subspaces $\l_0$ and $\l$  is the eventual difference between the dimensions of
$$
\l_0\cap\l^\perp \ \hbox{ and } \ \l_0^\perp\cap\l.
$$
There exists a geodesic joining $\l_0$ and $\l$ if and only if these dimensions coincide  (it is unique among minimal geodesics if and only if these dimensions are zero). 
Therefore, (if we consider $\l_0$ fixed) a closed subspace subspace $\l$ such that $\dim\l_0\cap\l^\perp=\dim\l_0^\perp\cap\l$, can be joined to $\l_0$ by a  geodesic $\delta$:  $\delta(0)=\l_0$ and $\delta(1)=\l$, which  is given by the action on $\l_0$ of a one parameter unitary group: $\delta$ is of the form 
$$
\delta(t)=e^{itZ}\l_0.
$$
The exponent $Z^*=Z$, which is co-diagonal with respect to the decomposition $\l_0\oplus\l_0^\perp$, and has norm $\|Z\|\le\pi/2$,   is factored by means of the Halmos'  decomposition. It is trivial in $(\l_0\cap\l)\oplus(\l_0^\perp\cap\l^\perp)$. It is given by
$$
Z'=\left( \begin{array}{cc} 0 & iX \\ -iX & 0 \end{array} \right)
$$
in the generic part $\h'=\l_0\times\l_0$. 
In the remaining part $\l_0\cap\l^\perp \oplus \l_o^\perp\cap\l$ it is not uniquely determined. In  the proof of Theorem \ref{37} below we recall how these multiple geodesics are obtained.

\begin{rem}
Consider the generic part $\h'$  of the fixed subspace $\l_0$ and a given subspace $\l$. There are two distinguished unitaries carrying $\l'_0$ to $\l'$ (the parts of $\l_0$ and $\l$ in $\h'$, respectively), namely the unitary $e^{iZ'}$ given by the unique geodesic and the symmetry $V$  given by Davis in \cite{davis}. This symmetry is obtained as the unitary part in the polar decomposition of $P_{\l_0'}+P_{\l'}-I_{\h'}$,
$$
P_{\l_0'}+P_{\l'}-I=V|P_{\l_0'}+P_{\l}'-I|=|P_{\l_0'}+P_{\l'}-I|V.
$$
$V$ is a symmetry because $P_{\l_0'}+P_{\l'}-I$ is a selfadjoint operator with trivial nullspace.
In terms of the operator $X$ given above, it is straightforward to verify that
$$
e^{iZ'}=\left(\begin{array}{cc} C & -S \\ S & C \end{array} \right) \ \hbox{ and } \ V=\left(\begin{array}{cc} C & S \\ S & -C \end{array} \right).
$$
They are related by 
$$
V=e^{iZ'}(2P_{\l_0'}-I)=(2P_{\l_0'}-I)e^{-iZ'}.
$$
This was proved in \cite{p-q}, and though it is a trivial verification, it is important in establishing the uniqueness of geodesics in the generic part. The following is also an easy verification.
\end{rem}
\begin{prop}
The unitary $e^{iZ'}$ commutes with the commutator $[P_{\l_0'},P_{\l'}]$. The symmetry $V$ anti-commutes with $[P_{\l_0'},P_{\l'}]$.
\end{prop}
The next result characterizes when the one-parameter unitary group $e^{itZ}$  remains inside $\sch_{\l_0}$. 

\begin{teo}\label{37}
Let  $\delta(t)=e^{itZ}\l_0$ be a  geodesic starting at
$\delta(0)=\l_0$.  Then the following are equivalent:
\begin{enumerate}
\item
$e^{it_0Z}\in\sch_{\l_0}$  for some $t_0\ne 0$.
\item
$e^{itZ}\in\sch_{\l_0}$  for all $t\in\mathbb{R}$.
\item
$Z'$, the generic part of $Z$, is diagonalizable.
\end{enumerate}
In this case   $e^{tZ}$ and
the commutators $[P_0,P_{\delta(t)}]$ are simultaneously diagonalizable,  for $t\in\mathbb{R}$.
\end{teo}
\begin{proof}
Fix $t\in\mathbb{R}$, $t\ne 0$.
In the generic part, $Z'$ is diagonalizable, and thus $e^{itZ}$ is diagonalizable. Since $e^{itZ}$ commutes with $[P_0,P_{\delta(t)}]$ (which acts non trivially only in the generic part), and both are diagonalizable, they can be simultaneously diagonalized. It remains to examine what happens in the non generic parts.  Denote  $\l=\delta(t)$.  The fact that $\l$ and $\l_0$ are joined by a geodesic means that $\dim \l_0\cap\l^\perp=\dim \l_0^\perp\cap\l$. It was proved in \cite{survey} that in the subspace
$$
\l_0\cap \l \ \oplus \ \l_0^\perp\cap \l^\perp,
$$
$Z$ is trivial and thus $e^{itZ}$ is the identity. In
$$
\h''=\l_0\cap \l^\perp \ \oplus \ \l_0^\perp\cap \l
$$
the (multiple) geodesics are constructed as follows \cite{survey}. Put
 $$
Z'':\h''=\l_0\cap \l^\perp \  \oplus \ \l_0^\perp \cap \l \to \h'' , \ Z''(\xi\oplus\eta)=-i\frac{\pi}{2}(W^*\eta\oplus -W\xi),
$$
where $W:\l_0\cap \l^\perp \to \l_0^\perp \cap \l$ is an arbitrary unitary isomorphism.  Clearly $Z''$ is selfadjoint. Let $\{\mu_n\}$ and $\{\nu_n\}$ be orthonormal bases of $\l_0\cap \l^\perp$ and $\l_0^\perp\cap \l$, respectively, and $W\mu_n=\nu_n$. Then 
$Z''\mu_n=i\frac{\pi}{2}\nu_n$ and $Z''\nu_n=-i\frac{\pi}{2}\mu_n$. Then, for any fixed $n$,  the subspace $\s_n$ generated by (the orthonormal) pair $\mu_n,\nu_n$ is stable under $Z''$. Clearly $Z''|_{\s_n}$ is selfadjoint, and therefore diagonalizable. Then $Z''=\oplus_{n\ge 1} Z''|_{\s_n}$ is diagonalizable in $\h''$. 
\end{proof}

\begin{rem}
Let us further digress upon  Example \ref{ejem 32}. It is easy to see that 
$$
\l_0=\{(a_k)\in\ell^2(\mathbb{Z}): a_k=a_{-k} \hbox{ for all } k\in\mathbb{Z}\}
$$
and $S\l_0$ ($S=$ the bilateral shift) are in generic position. Therefore, there exists a unique geodesic joining $\l_0$ ans $S\l_0$. It is determined by the operator $X$. From Halmos' model for a pair of projections in generic position \cite{halmos}, one has that
$$
\left( \begin{array}{cc} \cos^2(X) & 0 \\ 0 &  0 \end{array} \right) =P_{\l_0} P_{S\l_0} P_{\l_0}=\frac18 (1+\Pi)S(1+\Pi)S^*(1+\Pi).
$$
Pick $f$ in  $\l_0$ (i.e. $f(\bar{z})=f(z)$ a.e.). Then 
$$
P_{\l_0} P_{S\l_0} P_{\l_0}f(z)=\frac18(z^2+\bar{z}^2+2)(f(z)+f(\bar{z}))=\frac12 (Re(z^2)+1)f(z).
$$
Then $\cos(X)=M_{\frac{1}{\sqrt2} \sqrt{Re(z^2)+1}}$ (multiplication operator) and $X=M_{\arccos\left(\frac{1}{\sqrt2} \sqrt{Re(z^2)+1}\right)}$. Then,
$\|X\|$ equals $\arccos$ of the minimum of the function $\frac{1}{\sqrt2} \sqrt{Re(z^2)+1}$, i.e. $\|X\|=\pi/2$. It follows that the geodesic (or Finsler) distance between $\l_0$ and $S\l_0$ equals
$$
d(\l_0,S\l_0)=\pi/2.
$$
On the other hand, let $Y^*=Y$ be a logarithm for the  bilateral shift $S$ : $e^{iY}=S$. For instance, put 
$Y=M_{arg(z)}$, where $arg:\TT\to [-\pi,\pi)$ is usual argument.  Then 
$$
\gamma(t)=e^{itY}P_{\l_0}e^{-itY}
$$
is a smooth curve of projections joining $\gamma(0)=P_{\l_0}$ and $\gamma(1)=P_{S\l_0}$. Therefore its length is greater or equal than $\pi/2$. Note that 
$$
\dot{\gamma}(t)=e^{itY}iYP_{\l_0}e^{-itY}-e^{itY}P_{\l_0}e^{-itY}iY=ie^{itY}[Y,P_{\l_0}]e^{-itY}.
$$
Thus $\|\dot{\gamma}(t)\|=\|[Y,P_{\l_0}]\|$. Therefore
$$
\|[Y,P_{\l_0}]\|=\int_0^1\|\dot{\gamma}(t)\|dt=length(\gamma)\ge \pi/2.
$$
\end{rem}

\section{Non orthogonal projections}
Let $\s,\t\subset\h$ be closed subspaces such that $\s\dot{+}\t=\h$, where $\dot{+}$ means direct (non necessarily orthogonal) sum. Let us  relate the product of projections $P_\s P_\t P_\s $ with the non orthogonal projection $Q=P_{\s\parallel\t}$ with range $\s$ and nullspace $\t$. 

One has the known formulas (see \cite{ando})
$$
P_\s=Q(Q+Q^*-I)^{-1} \ \hbox{ and } \ P_\t=(I-Q)(I-Q-Q^*)^{-1}.
$$
Note that the selfadjoint invertible operator $A=Q+Q^*-I$ satisfies $AQ=Q^*A$ and $AQ^*=QA$. Then
$$
P_\s P_\t=-Q(I-Q^*)A^{-2}.
$$
On the other hand, writing operators as matrices in terms of the decomposition $\h=\s\oplus\s^\perp$, 
$$
P_\s=\left( \begin{array}{cc} 1 & 0 \\ 0 & 0 \end{array}\right) , \hbox{ and } \  Q=\left( \begin{array}{cc} 1 & B \\ 0 & 0 \end{array}\right), 
$$ 
where $B=P_\s Q|_{\s^\perp}:\s^\perp\to \s$. Note that
$$
A^2=\left( \begin{array}{cc} 1+BB^* & 0 \\ 0 & 1+B^*B \end{array}\right), 
$$
and thus
$$
P_\s P_\t=\left( \begin{array}{cc} BB^*(1+BB^*)^{-1} & -B(1+B^*B)^{-1} \\ 0 & 0 \end{array}\right)
\hbox{ and } \
P_\s P_\t P_\s=\left( \begin{array}{cc} BB^*(1+BB^*)^{-1} & 0 \\ 0 & 0 \end{array}\right).
$$
In particular, we obtain the following result
\begin{prop}\label{lema 41}
$P_\s P_\t$ is Schmidt decomposable if and only if $P_{\s\parallel \t}|_{\s^\perp}=B:\s^\perp\to \s$ is Schmidt decomposable. In this case, the singular values $s_n$ of $P_\s P_\t$ and the singular values $\beta_n$ of $B$ are related by 
$$
s_n=\frac{\beta_n}{\sqrt{\beta_n^2+1}}.
$$
\end{prop}

The decreasing order is preserved because the map $f(t)=\frac{t}{\sqrt{t^2+1}}$ is strictly increasing.

In \cite{cpr}, Corach, Porta and Recht studied the geometry of the fibration from the space the oblique (non  orthogonal) projections  onto the space of orthogonal projections, in the setting of arbitrary C$^{*}$-algebras. There are several ways to assign an
orthogonal projection to a non orthogonal one: the projection onto the range, the projection onto the nullspace, etc. But it is this fibration of Corach, Porta and Recht, based on the polar decomposition, that has remarkable metric properties. Let us describe this map.

Given $Q\in\b(\h)$ with $Q^2=Q$, consider the reflection $2Q-1$, which satisfies $(2Q-I)^2=I$, is equal to the identity in $R(Q)$ and minus the identity in $N(Q)$ . Let
$$
2Q-I=\rho_Q |2Q-I|
$$
be the polar decomposition. In \cite{cpr} it was  proved that
\begin{enumerate}
\item
$\rho_Q$ is a symmetry: $\rho_Q^2=I$, $\rho_Q^*=\rho_Q$.
\item
$\rho_Q|2Q-I|=|2Q-I|^{-1}\rho_Q$.
\end{enumerate}

Let  $Q$ be a non orthogonal projection onto $\l_0$. We shall consider  the question of when  $\rho_Q\in\sch_{\l_0}$. Recall from example \ref{ejemplos}.3, that a symmetry $\rho_Q\in\sch_{\l_0}$ if and only if there exist bi-orthogonal bases of $\l_0$ and  the subspace $N(\rho_Q-I)$. Note that
$$
N(\rho_Q-I)=\{\xi\in\h: \rho_Q\xi=\xi\}=\{\xi\in\h: (2Q-I)\xi=|2Q-I|\xi\},
$$
where the last assertion follows from the algebraic properties of $2Q-I$ and $\rho_Q$

In matrix form, in terms of the decomposition $\h=\l_0\oplus \l_0^\perp$, 
$$
Q=\left(\begin{array}{cc}
1 & A \\
0 & 0
\end{array}\right) \  \!\!\!,
\ 2Q-I=\left(\begin{array}{cc}
1 & 2A \\
0  & -1 \end{array}\right)
\ \hbox{ and }
\  |2Q-I|^2=\left(\begin{array}{cc} 1 & 2A \\ 2A^*  & 4A^*A+1 \end{array}\right).
$$
Note that we are interested in the $1,1$ entry of the square root of the last (positive) matrix. Indeed,
\begin{lem}\label{lema 1,1}
$\rho_Q\in\sch_{\l_0}$ if and only if the $1,1$ entry of $|2Q-I|$ in the matrix in terms of the decomposition $\l_0\oplus\l_0^\perp=\h$ is Schmidt decomposable.
\end{lem}
\begin{proof}
$\rho_Q\in\sch_{\l_0}$ if and only if $P_0\rho_QP_0$ is Schmidt decomposable.
$$
P_0\rho_QP_0=P_0|2Q-I|(2Q-I)P_0=P_0|2Q-I|P_0,
$$
because $R(Q)=\l_0$, and thus $(2Q-I)P_0=P_0$.
\end{proof}

\begin{rem} 

One can write $Q$ in terms of the Halmos decomposition induced by the subspaces $\l_0=R(Q)$ and $N(Q)$.
Clearly $\l_0\cap N(Q)=\{0\}$ and $\l_0^\perp\cap N(Q)^\perp=(\l_0+N(Q))^\perp=\{0\}$. On $\l_0\cap N(Q)^\perp$, $Q$ is the identity, and on $\l_0^\perp\cap N(Q)$, $Q$ is trivial.
In \cite{buckholtz}, D. Buckholtz proved that two closed subspaces $\s$, $\t$ satisfy that $\s\dot{+}\t=\h$ if and only if $P_\s-P_\t$ is invertible. In our case, this implies that $P_0-P_{N(Q)}$ is invertible. Also in \cite{buckholtz},   the formula
$$
Q=P_0\left(P_0-P_{N(Q)}\right)^{-1}
$$
was established. In the generic subspace (between $\l_0$ and $N(Q)$), we have therefore (denoting by $Q'$ the restriction of $Q$ to this part)
$$
Q'=\left(\begin{array}{cc} 1 & 0 \\ 0 & 0 \end{array}\right)\left(\begin{array}{cc} 1-C^2 & -CS \\ -CS & -S^2 \end{array}\right)^{-1}.
$$  
The operator $(P_0-P_{N(Q)})^2=\left(\begin{array}{cc} S^2 & 0 \\ 0 & S^2 \end{array}\right)$  is invertible. Then $S$ is invertible, and therefore
$$
(P_0-P_{N(Q)})^{-1}=(P_0-P_{N(Q)})(P_0-P_{N(Q)})^{-2}=\left(\begin{array}{cc} 1 & -CS^{-1} \\ -CS^{-1} & -1 \end{array}\right),
$$
and
$$
Q'=\left(\begin{array}{cc} 1 & -CS^{-1} \\ 0 & 0 \end{array}\right).
$$
\end{rem}
In particular, one obtains that
\begin{coro}
Let $Q$ be a projection. Then $Q$ is unitarily equivalent to a projection $Q_+$ with matrix (in terms of the decomposition $\h=R(Q_+)\oplus N(Q_+)^\perp$) 
$$
Q_+=\left(\begin{array}{cc} 1 & B \\ 0 & 0  \end{array}\right),
$$
with $B\ge 0$.
\end{coro}
\begin{proof}
$Q'$ is unitarily equivalent to its Halmos model, which has a negative $1,2$ entry. It suffices to conjugate this model with the symmetry $\left(\begin{array}{cc} 1 & 0 \\ 0 & -1 \end{array}\right)$. On the other (non generic parts), $Q$ is either trivial or the identity.
\end{proof}

Also, writing $Q$ in the Halmos decomposition induced by $R(Q)=\l_0$ and $N(Q)$ allows us to obtain a formula for $|2Q-I|$. 
Note that in the generic part $\h'$, one has 
$$
|2Q-1|^2=\left( \begin{array}{cc}
1 & -2CS^{-1} \\
-2CS^{-1} & 4 C^2S^{-2}+1
\end{array}\right)=
S^{-2}\left( \begin{array}{cc}
S^2 & -2CS \\
-2CS & 3 C^2+1
\end{array}\right).
$$
\begin{lem}
With the current notations
\begin{equation}\label{diagonalizacion}
\left( \begin{array}{cc}
S^2 & -2CS \\-2CS & 3 C^2+1
\end{array}\right)=
U \left( \begin{array}{cc}
(1+C)^2 & 0 \\
0 & (1-C)^2 \end{array}\right) U^*,
\end{equation}
where
$$
U=\frac{1}{\sqrt{2}} \left( \begin{array}{cc} -S(1+C)^{-1/2} & S(1-C)^{-1/2} \\(1+C)^{1/2} & (1-C)^{1/2} \end{array}\right)
$$
is a unitary operator in $\h_0$.
\end{lem}
\begin{proof}
Note that $(1-C)(1+C)=S^2$ is invertible, thus $1+C$ and $1-C$ are positive and invertible, and the matrices above make sense. The proof that $U$ is unitary and that the factorization holds are straightforward verifications. This factorization is obtained by formally diagonalizing the matrix of $|2Q-1|^2$ ($(1\pm C)^2$ are its formal eigenvalues).
\end{proof}
Putting these formulas together,
\begin{coro}
Let $Q$ be a (possibly oblique) projection with $R(Q)=\l_0$.
In the Halmos decomposition given by $\l_0=R(Q)$ and $N(Q)$, the modulus $|2Q-I|$ of $2Q-I$ is given by the identity
$I_{\l_0\cap N(Q)^\perp\oplus \l_0 \cap N(Q)}$  in the non trivial non generic part $\l_0\cap N(Q)^\perp\oplus \l_0 \cap N(Q)$. In the generic part $\h_0$ it is given by
$
|2Q_0-I|= \left( \begin{array}{cc} S & -C \\ -C & (1+C^2)S^{-1} \end{array} \right).
$
\end{coro}
\begin{proof}
The assertion on the non generic part is clear. In $\h_0$
$$
|2Q_0-I|=\left(S^{-2}U \left( \begin{array}{cc} (1+C)^2 & 0 \\ 0 & (1-C)^2 \end{array} \right)U^*\right)^{1/2}=S^{-1} U  \left( \begin{array}{cc} 1+C & 0 \\ 0 & 1-C \end{array} \right)U^*
$$
$$
= \left( \begin{array}{cc} S & -C \\ -C & (1+C^2)S^{-1} \end{array} \right).
$$
\end{proof} 
Returning to our original question, the answer is now straightforward:
\begin{teo}
Let $Q=\left(\begin{array}{cc} 1 & A \\ 0 & 0 \end{array}\right)$ be a projection with $R(Q)=\l_0$. Then  the following are equivalent:
\begin{enumerate}
\item
$\rho_Q\in\sch_{\l_0}$.
\item
The operator $X$ of the Halmos model induced by $R(Q)$ and $N(Q)$ is diagonalizable.
\item
$QQ^*$ is diagonalizable (or equivalently, $Q^*Q$ is diagonalizable).
\item
$A$ is Schmidt decomposable.
\end{enumerate}
\end{teo}
\begin{proof}
By Lemma \ref{lema 1,1}, $\rho_Q\in\sch_{\l_0}$ if and only if the $1,1$ entry of $|2Q-I|$ is diagonalizable, which is equivalent to $S=\sin(X)$ being  diagonalizable. This happens  if and only if $X$ is diagonalizable. Note that $QQ^*=\left(\begin{array}{cc} 1+AA^* & 0 \\ 0 & 0 \end{array}\right)$. Using Halmos model, one gets that $QQ^*$ is unitarily equivalent to
$$
I \oplus I \oplus \left(\begin{array}{cc} S^{-2} & 0 \\ 0 & 0 \end{array}\right)
$$
in the orthogonal decomposition $\l_0\cap N(Q)^\perp \oplus \l_0^\perp \cap N(Q) \oplus \h_0$.
It follows that $S$ is diagonalizable if and only if $QQ^*$ is diagonalizable, which happens if and only if $AA^*$ is diagonalizable. 
\end{proof}

\begin{rem}
One can compute the form of $\rho_Q$ in the Halmos decomposition (induced by $R(Q)=\l_0$ and $N(Q)$.
In the non generic part $\l_0\cap N(Q)^\perp \oplus \l_0^\perp \cap N(Q)$, $\rho_Q$ is given by $I\oplus -I$.
In the generic part $\h_0$ it is given by
$$
\rho_Q|_{\h_0}=|2Q_0-I|(2Q_0-I)=\left(\begin{array}{cc} S & -C \\ -C & (1+C^2)S^{-1} \end{array}\right)\left(\begin{array}{cc} 1 & -2CS^{-1} \\ 0 & -1 \end{array}\right)=\left(\begin{array}{cc} S & -C \\ -C & -S \end{array}\right).
$$
\end{rem}

\section{Corners and dilations of contractions}
An arbitrary contraction $A:\l_0\to\l_0$ can be obtained as the $1,1$ entry of a unitary operator in a bigger Hilbert space. 
We start by considering the following elementary construction in $\l_0\times\l_0$, which is well known, and was described, for instance, in \cite{halmos brasil},
$$
V_A=\left( \begin{array}{cc} A & (1-AA^*)^{1/2} \\ (1-A^*A)^{1/2} & -A^* \end{array} \right).
$$
First note that if $A=A^*$, then $V_A$ is selfadjoint, i.e. a symmetry, and thus diagonalizable. This means that the condition $V_A$ diagonalizable does not imply any decomposition property for $A$. We have the following:
\begin{prop}\label{jva}
$A$ is Schmidt-decomposable if and only $\left(\begin{array}{cc} 0 & 1 \\ 1 & 0 \end{array} \right) V_A$ is diagonalizable.
\end{prop}
\begin{proof}
Denote $J=\left(\begin{array}{cc} 0 & 1 \\ 1 & 0 \end{array} \right)$ and suppose that  $JV_A$ is diagonalizable. Note that 
\begin{equation}\label{matriz jva}
J V_A
=\left(\begin{array}{cc} (1-A^*A)^{1/2} & 0 \\ 0 & (1-AA^*)^{1/2} \end{array} \right)+\left(\begin{array}{cc} 0 & -A^* \\ A & 0 \end{array} \right),
\end{equation}
where the left hand matrix is selfadjoint and the right hand matrix is anti-selfadjoint.  They are the real and imaginary parts of  $JV_A$, which is a unitary operator, and therefore normal. Therefore these two terms above commute. Since $JV_A$ is diagonalizable, its  real and imaginary parts are simultaneously diagonalizable. Thus, in particular, $(1-A^*A)^{1/2}$ is diagonalizable in $\l_0$. Then $A^*A$ is diagonalizable, and $A$ has a singular value decomposition.

Conversely, let $A=\sum_{n\ge 1} s_n \eta_n\otimes\xi_n$ be a singular value decomposition for $A$, with $\{\eta_n\}$ and $\{\xi_n\}$ orthonormal systems, which span $N(A)^\perp$ and $\overline{R(A)}=N(A^*)^\perp$, respectively. On
$$
N(A)\times \{0\} \ \oplus \ \{0\} \times N(A^*),
$$
the operator $JV_A$ is the identity. So it suffices to consider $JV_A$ on the orthogonal complement of this subspace (in $\l_0\times\l_0$), which is clearly an invariant subspace for $JV_A$. The complement of this subspace is $N(A)^\perp\times N(A^*)^\perp$. For each $n\ge 1$, denote by $\mathbb{S}_n$ the $2$-dimensional space generated by the (orthonormal) pair $v_n=\left(\begin{array}{c} \eta_n \\ 0 \end{array}\right)$,  $w_n=\left(\begin{array}{c} 0 \\ \xi_n \end{array}\right)$. Note that 
$$
JV_Av_n=(1-s_n^2)^{1/2}v_n  +s_n w_n \  \hbox{ and } \ JV_Aw_n=-s_n v_n +(1-s_n^2)^{1/2} w_n.
$$ 
That is, $\mathbb{S}_n$ is invariant for $JV_A$. Thus $JV_A$ is diagonalizable in each block $\mathbb{S}_n$, and therefore also on $\oplus_{n\ge 1} \mathbb{S}_n=N(A)^\perp\times N(A^*)^\perp$.
\end{proof}
\begin{rem}\label{52}
In the above situation ($A$ decomposable with singular values $s_n$), the eigenvalues of of $JV_A$ in $N(A)^\perp\times N(A^*)^\perp$ are $(1-s_n^2)^{1/2}+ i\ s_n$ and  $(1-s_n^2)^{1/2} - i\ s_n$, each one with the same multiplicity as $s_n$.
\end{rem}
Let us consider  now the unitary dilation constructed by B. Sz-Nagy and C. Foias \cite{nagyfoias}. We recall this construction. As above, let $A$ be a contraction in $\l_0$. Consider the Hilbert space
$$
{\bf H}:=\bigoplus_{n\in\mathbb{Z}} \l_n,
$$
where $\l_n$ is a copy of $\l_0$, and the subspace $\l_0\subset {\bf H}$ stands as the center summand ($n=0$). Then every operator $T$  in ${\bf H}$ can be regarded as a matrix $(T_{i,j})$, $i,j\in\mathbb{Z}$.  The dilation by Nagy and Foias is the unitary operator $U_A=(U_{i,j})$,  whose matrix entries are given by
$$
U_{0,0}=A , \ U_{0,1}=D_{A^*}, \ U_{-1,0}=D_A , \ U_{-1,1}=-A^*, 
$$
$U_{i,i+1}=I_{\l_0}$ for $i\ne 0,1$, and $U_{i,j}=0$ for all other $i,j\in \mathbb{Z}$, where $D_T$ denotes the defect operator $(1-T^*T)^{1/2}$.

Denote by ${\bf S}$ the bilateral  shift (with multiplicity $\dim \l_0$) of ${\bf H}$: 
$$
{\bf S}\sum_{j\in\mathbb{Z}}\xi_j=\sum_{j\in\mathbb{Z}} \xi_{j+1}.
$$ 

\begin{teo}
Let $A$ be a contraction in $\l_0$, and $U_A$  the Nagy-Foias dilation  of $A$. Then $A$ is Schmidt decomposable  if and only if
${\bf S}U_A$ is diagonalizable.
\end{teo}
\begin{proof}
An elementary matrix computation, shows that ${\bf S}U_A$, in the decomposition
$$
{\bf H}=\bigoplus_{j<0}\l_j \oplus \left(\l_0\oplus\l_1\right)  \oplus \bigoplus_{j>1}\l_j
$$
has the block-diagonal form
$$
{\bf S}U_A=I\oplus N_A \oplus I,
$$
where $N_A=\left(\begin{array}{cc} D_A & -A^* \\ A & D_{A^*} \end{array} \right)$. Note that $N_A$ coincides with $JV_A$ in (\ref{matriz jva}). Thus the proof follows applying Proposition \ref{jva}.
\end{proof}
 
\section{Multiplication by continuous unimodular  functions}

Recall Example \ref{ejemplos}.1: $\h=L^2(\mathbb{T})$, $\l_0=H^2(\mathbb{T})$ and $U=M_\varphi$ for $\varphi:\mathbb{T}\to\mathbb{T}$ continuous. As seen $M_\varphi\in\sch_{H^2(\mathbb{T})}$. In particular, to $\varphi$ corresponds a sequence $\{s_n(\varphi)=s_n\}$ of real numbers $0<s_n\le 1$, namely, the singular values of the Toeplitz operator with symbol $\varphi$. 

 If $\varphi$ is analytic in $\mathbb{D}$, then it must be a finite Blaschke product. In particular $P_0U|_{\l_0}=M_\varphi|_{H^2(\mathbb{T})}$ is an isometry in $H^2(\mathbb{T})$. Then its singular values are  the sequence $s_n=1$.

The first non trivial case would be to consider a rational continuous unimodular function, i.e. $\varphi=B_\aa / B_\bb$, where $\aa=\{a_1,\dots, a_n\}$ and $\bb=\{b_1,\dots, b_m\}$ are finite sequences of zeros, and 
$B_\aa$ , $B_\bb$ are the corresponding Blaschke products,
$$
B_\aa(z)=\prod_{j=i}^n \frac{z-a_j}{1-\bar{a}_j z}, \ B_\bb=\prod_{k=i}^m \frac{z-b_k}{1-\bar{b}_k z}.
$$
Assume that  $n=m$, $a_j\ne a_k$ if $j\ne k$, and the same for $\bb$ (w.l.o.g. $a_j\ne b_k$). 
Denote $\aa\bb=\{a_1,\dots, a_n,b_1,\dots,b_n\}$. We want to characterize the singular values of $T_\varphi$, or equivalently, of $P_{\l_0}M_\varphi P_{\l_0}M_{\varphi^{-1}}$. Note that 
$$
P_{\l_0}M_\varphi P_{\l_0} M_{\varphi^{-1}}=M_{B_\bb^{-1}}\{ M_{B_\bb} P_{\l_0} M_{B_\bb^{-1}} M_{B_\aa}P_{\l_0}M_{B_\aa^{-1}}\} M_{B_\bb},
$$
and thus one can compute the singular values of 
$$
M_{B_\bb} P_{\l_0} M_{B_\bb^{-1}} M_{B_\aa}P_{\l_0}M_{B_\aa^{-1}}= P_{B_\bb\l_0} P_{B_\aa\l_0}.
$$
Clearly, we  need to compute  the generic part of the subspaces $B_\aa\l_0$ and $B_\bb\l_0$. 
Clearly $\h_{11}=B_{\aa\bb}\l_0$, and
$$
\h_{00}=B_\aa \l_0^\perp \cap B_\bb \l_0^\perp= \{B_\aa \l_0 \vee B_\bb \l_0\}^\perp= (H^2(\TT))^\perp=\l_0^\perp,
$$ 
because $B_\aa$ and $B_\bb$ are {\it co-prime} inner functions (see \cite{garciaross}).
Then
$$
\h_{01}=\h_{10}=\{0\}.
$$
Indeed, if $f\in B_\aa\cap B_\bb^\perp$, $f=g B_\aa=\sum_{j=1}^n \beta_j c_{b_j}$, where $\beta_j\in\mathbb{C}$ and  $c_b$ denotes the Szego kernel 
$$
c_b(z)=\displaystyle{\frac{1}{1-\bar{b}z}}.
$$
Thus $f$ can be written as a rational function $f(z)=\frac{p(z)}{\prod_{j=1}^n (1-\bar{b}_jz)}$, with $p$ of degree $n-1$. On the other hand, $f$ has $n$ different zeros, and thus $f=0$.
$$
\h'=H^2(\TT)\ominus B_{\aa\bb} H^2(\TT),
$$ 
the model space, usually denoted $\k_{B_{\aa\bb}}$. The space $\k_{B_{\aa\bb}}$ is generated by the (non orthogonal) functions Szego kernels 
$c_{a_j}, c_{b_k}$, $j,k=1,\dots, n$.

The reduction $\h'_\aa$ of $B_\aa\l_0$ to the generic part $\h'=\k_{\aa\bb}$ consists of the functions with vanish at $a_1,\dots,a_n$.
Therefore they are orthogonal to $c_{a_1}(z),\dots,c_{a_n}(z)$.  Thus 
$$
(\h'_\aa)^\perp=\langle c_{a_j}: j=1,\dots,n\rangle=\k_\aa \ \hbox{ and } \ (\h'_\bb)^\perp=\langle c_{b_k}: k=1,\dots,n\rangle=\k_\bb .
$$

Note that the singular values of $P_{\h'_\aa}P_{\h'_\bb}$ are strictly less than one . Indeed, in the generic part, since the intersection of the subspaces is trivial, the singular values are strictly less than $1$ (see Remark \ref{pq}.1). Therefore we can consider instead the singular values of $P^\perp_{\h'_\aa}P^\perp_{\h'_\bb}$.

For instance
\begin{ejem}
If  $n=2$ (and $\dim\k_{\aa\bb}=4$), $\h_\aa^\perp=\langle c_{a_1},c_{a_2}\rangle$, $\h_\bb^\perp=\langle c_{b_1},c_{b_2}\rangle$, the squares of these singular values are the eigenvalues of the (symmetric) matrix
$$
\left( \begin{array}{cc} \displaystyle{\frac{|\langle u_1, v_1\rangle |^2}{\|u_1\|^2\|v_1\|^2}+  \frac{|\langle u_1, v_2\rangle |^2}{\|u_1\|^2\|v_2\|^2}}   &  
\displaystyle{\frac{\langle u_1, v_1\rangle \langle v_1,u_2\rangle}{\|v_1\|^2\|u_1\|\|u_2\|}+\frac{\langle u_1, v_2\rangle \langle v_2, u_2\rangle}{\|v_2\|^2\|u_1\|\|u_2\|}} \\ 
\displaystyle{\frac{\langle v_1, u_1\rangle \langle u_2, v_1\rangle}{\|v_1\|^2\|u_1\|\|u_2\|}+\frac{\langle v_2, u_1\rangle \langle u_2, v_2\rangle}{\|v_2\|^2\|u_1\|\|u_2\|}}  & 
\displaystyle{ \frac{|\langle u_2, v_1\rangle |^2}{\|u_2\|^2\|v_1\|^2}+  \frac{|\langle u_2, v_2\rangle |^2}{\|u_2\|^2\|v_2\|^2}} \end{array} \right)
$$
where $u_1(z)=c_{a_1}(z)$ and $u_2(z)=c_{a_2}(z)-\frac{1-|a_1|^2}{1-\bar{a}_1a_2}c_{a_1}(z)$ is the Gram-Schmidt orthonormalization of $c_{a_1}, c_{a_2}$, and similarly $v_1(z)=c_{b_1}(z)$, and $v_2(z)=c_{b_2}(z)-\frac{1-|b_1|^2}{1-\bar{b}_1b_2}c_{b_1}(z)$ is the orthonormalization of $c_{b_1}, c_{b_2}$. Note that the inner products (and norms) involving the Szego kernels can be explicitly computed. 
\end{ejem}

In general, one can say the following:
\begin{prop}
Let $\aa$ and $\bb$ as above, and  
$Q=P_{(\h_\aa')^\perp\parallel (\h_\bb')^\perp}$, which is the (non orthogonal) projection  
$$
\k_{\aa\bb}=\langle c_{a_1},\dots,c_{a_n}\rangle \oplus \langle c_{b_1},\dots,c_{b_n}\rangle\to \langle c_{a_1},\dots,c_{a_n}\rangle\subset \k_{\aa\bb} \ \ , f\oplus g \mapsto f.
$$
Then the singular values $s_n$ of $P_0M_{B_\aa / B_\bb}|_{\l_0}$ which are strictly less than one are finite, and  are 
$$
s_n=\displaystyle{\frac{\beta_n}{\sqrt{\beta_n^2+1}}} \ , \ n=1,\dots,n
$$
where $\beta_n$ are the singular values of $Q$.
\end{prop}
\begin{proof}
Use Lemma \ref{lema 41}
\end{proof}

Esteban Andruchow \\
Instituto de Ciencias,  Universidad Nacional de Gral. Sarmiento,
\\
J.M. Gutierrez 1150,  (1613) Los Polvorines, Argentina
\\ 
and
\\
 Instituto Argentino de Matem\'atica, `Alberto P. Calder\'on', CONICET, 
\\
Saavedra 15 3er. piso,
(1083) Buenos Aires, Argentina.
\\
e-mail: eandruch@ungs.edu.ar


\begin{thebibliography}{10}

\bibitem{ando}
T.~Ando,
\newblock Unbounded or bounded idempotent operators in Hilbert space,
\newblock {\em Linear Algebra Appl.} 438 (2013), 3769--3775.

\bibitem{p-q}
E.~Andruchow,
\newblock Operators which are the difference of two projections,  
\newblock {\em J. Math. Anal. Appl.} 420 (2014),  1634--1653.



\bibitem{survey}
E.~Andruchow,
\newblock The Grassmann manifold of a Hilbert space,
\newblock {\em J. Proceedings of the XIIth "Dr. Antonio A. R. Monteiro'' Congress}, Univ. Nac. del Sur, Bahía Blanca, 41--55, 2014.


\bibitem{grassH2}
E.~Andruchow, E.~Chiumiento, and G.~Larotonda,
\newblock Geometric significance of {T}oeplitz kernels,
\newblock {\em J. Funct. Anal.}, 275 (2018), 329--355.

\bibitem{pqvsp-q}
E.~Andruchow and G.~Corach,
\newblock Schmidt decomposable products of projections,
\newblock {\em Integral Equations Operator Theory}, 89 (2017), 557--580.

\bibitem{jot}
E.~Andruchow and G.~Corach,
\newblock Essentially orthogonal subspaces,
\newblock {\em J. Operator Theory}, 79 (2018), 79--100.

\bibitem{ass} J. Avron, R.  Seiler and B. Simon,  The index of a pair of projections, {\em  J. Funct. Anal.} 120 (1994),  220--237.

\bibitem{buckholtz}
D.~Buckholtz,
\newblock Hilbert space idempotents and involutions,
\newblock {\em Proc. Amer. Math. Soc.} 128 (2000), 1415--1418. 

\bibitem{cpr}
G.~Corach, H.~Porta, and L.~Recht,
\newblock The geometry of spaces of projections in {$C^*$}-algebras,
\newblock {\em Adv. Math.}, 101 (1993), 59--77.

\bibitem{davis}
C.~Davis,
\newblock Separation of two linear subspaces,
\newblock {\em Acta Sci. Math. Szeged}, 19 (1958), 172--187.

\bibitem{folland}
G.~B.~Folland and A.~Sitaram,
\newblock The uncertainty principle: a mathematical survey,
\newblock {\em J. Fourier Anal. Appl.}, 3 (1997), 207--238.

\bibitem{garciaross}
S.~R.~Garcia and W.~T.~Ross,
\newblock Model spaces: a survey,
\newblock  In {\em Invariant subspaces of the shift operator}, volume 638 of
  {\em Contemp. Math.}, pages 197--245. Amer. Math. Soc., Providence, RI, 2015.


\bibitem{halmos brasil} 
P.~R.~Halmos,
\newblock  Normal dilations and extensions of operators,
\newblock {\em Summa Brasil. Math.}, 2 (1950). 125--134.



\bibitem{halmos}
P.~R.~Halmos,
\newblock Two subspaces,
\newblock {\em Trans. Amer. Math. Soc.}, 144 (1969), 381--389.


\bibitem{hartman}
P. ~Hartman,
\newblock On completely continuous {H}ankel matrices,
\newblock {\em Proc. Amer. Math. Soc.}, 9 (1958), 862--866.

\bibitem{landau} H.J. Landau, Necessary density conditions for sampling and interpolation of certain entire functions, {\em Acta Math.} 117 (1967), 37--52. 

\bibitem{nagyfoias}
B.~Sz.-Nagy, C.~Foias, H.~ Bercovici and L.~K\'erchy,
\newblock {\em Harmonic analysis of operators on Hilbert space}. Second edition. Revised and enlarged edition. Universitext. Springer, New York, 2010. 


\bibitem{pr}
H.~Porta and L.~Recht,
\newblock Minimality of geodesics in {G}rassmann manifolds,
\newblock {\em Proc. Amer. Math. Soc.}, 100 (1987), 464--466.

\bibitem{schmidt}
E.~Schmidt,
\newblock Zur Theorie der linearen und nichtlinearen Integralgleichungen,
\newblock {\em Math. Ann.}, 63 (1907), 433--476.

\bibitem{segalwilson} 
G. Segal, G. Wilson,  Loop groups and equations of KdV type, {\em Inst. Hautes  \'Etudes Sci. Publ. Math.}  61 (1985), 5–65.

\end{thebibliography}
\end{document}